\documentclass[12 pt]{article}%
\usepackage{amsmath, amsfonts, amsthm, color,latexsym}
\usepackage{amsmath}
\usepackage{amsfonts}
\usepackage{amssymb}
\usepackage{color, soul}
\usepackage[all]{xy}
\usepackage{graphicx}%
\providecommand{\U}[1]{\protect\rule{.1in}{.1in}}
\allowdisplaybreaks[4]
\newtheorem{theorem}{Theorem}[section]
\newtheorem{proposition}[theorem]{Proposition}
\newtheorem{corollary}[theorem]{Corollary}
\newtheorem{example}[theorem]{Example}
\newtheorem{examples}[theorem]{Examples}
\newtheorem{remark}[theorem]{Remark}

\newtheorem{lemma}[theorem]{Lemma}
\newtheorem{final remark}[theorem]{Final Remark}
\newtheorem{definition}[theorem]{Definition}
\allowdisplaybreaks[4]

\begin{document}

\title{Complete latticeability in vector-valued sequence spaces}
\author{Geraldo Botelho\thanks{Supported by CNPq Grant
304262/2018-8 and Fapemig Grant PPM-00450-17.}\,\, and  Jos\'e Lucas P. Luiz\thanks{Supported by a CNPq scholarship.\newline 2020 Mathematics Subject Classification: 46B42, 46B87, 46G25, 46B45.\newline Keywords: Banach lattices, regular polynomials, complete latticeability, positive polynomial Schur property.
}}
\date{}
\maketitle

\begin{abstract} First we adjust a technique due to Jim\'enez-Rodr\'iguez to prove the complete latticeability of the set of disjoint non-norm null weakly null sequences  and of the set of disjoint non-norm null regular-polynomially null sequences in Banach lattices. Then we apply the mother vector technique to prove the complete latticeability of $\lambda_\pi(E) \setminus \lambda_s(E)$, which implies the complete latticeability of $(\ell_p\widehat{\otimes}_{|\pi|}E)\setminus(\ell_p\widehat{\otimes}_{\pi}E)$, where $E$ is a Banach lattice and $1 < p < \infty$.
\end{abstract}

\section{Introduction}

In this paper we give a contribution to the fashionable subject of lineability, which is the search for linear structure inside nonlinear environments. The book \cite{aron1} is a very good reference for the state of the art in lineability. Recently, Oikhberg \cite{oikhberg} pushed the Banach lattice setting into this subject and coined the following terms: a subset $A$ of a Banach lattice is {\it latticeable} ({\it completely latticeable}) if there exists a (closed) infinite dimensional sublattice of $E$ all of whose elements but the origin belong to $A$ (see also \cite{oikhberg2}).

Our contribution to latticeability is splitted into two sections. Among tons of lineability-type results which have appeared in the last years, our focus in Section \ref{section2} is the following result (actually its proof) proved by Jim\'enez-Rodr\'iguez \cite{jimenez}: if $E$ is a Banach space failing the Schur property, then the set of non-norm null weakly null $E$-valued sequences contains, except for the origin, a closed infinite-dimensional subspace of $c_0^w(E)$, which is the closed subspace of $\ell_\infty(E)$ formed by weakly null sequences. We solve in Section \ref{section2} the following three questions that arise naturally from Jim\'enez-Rodr\'iguez' result (precise definitions will be given in due time): \\
(i) If $E$ is a non-polynomially Schur Banach space, then the set of non-norm null polynomially null sequences contains, up to the origin, a closed infinite dimensional subspace of $c_0^w(E)$?\\
(ii) If $E$ is a Banach lattice failing the positive Schur property, then the set of disjoint non-norm null weakly null $E$-valued sequences is latticeable in $\ell_\infty(E)$? Completely latticeable?\\
(iii) If $E$ is a non-positively polynomially Schur Banach lattice, then the set of disjoint non-norm null regular-polynomially null $E$-valued sequences is latticeable in $\ell_\infty(E)$? Completely latticeable?

Let $E$ be a Banach lattice and let $\lambda$ be a scalar-valued sequence space. In Section \ref{section3} we turn our attention to the Banach lattice $\lambda_\pi(E)$ of $E$-valued sequences and its subspace $\lambda_s(E)$ introduced in \cite{bu2} and developed in, e.g., \cite{botelho1,bu1,bu2,bu5,bu3}. We use the mother vector technique from \cite{geraldo} to prove that, whenever nonempty,  $\lambda_\pi(E)\setminus \lambda_s(E)$ is completely latticeable. As consequences we prove, for $1 < p < \infty$, the complete latticeability of the set of sequences in $\ell_p^\pi(E)$ that are not Cohen strongly $p$-summable, a related result for Orlicz vector-valued sequence lattices and the latticeability of the complement of the projective tensor product $\ell_p\widehat{\otimes}_{\pi}E$ in the positive projective tensor product $\ell_p\widehat{\otimes}_{|\pi|}E$.

All sequence spaces in this paper are considered as Banach lattices with the coordinatewise order. By $B_E$ we denote the closed unit ball of the Banach space $E$. For the general theory of Banach lattices we refer to \cite{aliprantis2, meyer}, for regular homogeneous polynomials in Banach lattices we refer to \cite{bu, loane2, loane1}, and for sequence spaces we refer to \cite{gupta,lindenstrauss}. By a disjoint sequence in a Riesz space we mean a pairwise disjoint sequence.

\section{Non-norm null disjoint sequences}\label{section2}

Some terminology is needed to describe and state the results we prove in this section.

Polynomially Schur spaces were introduced by Carne, Cole and Gamelin \cite{carne} and have been developed by several authors (see, e.g, \cite{aron, arranz, farmer,jaramillo}). A sequence $(x_j)_{j=1}^\infty$ in a Banach space $E$  is {\it polynomially null} if $P(x_j) \longrightarrow 0$ for every scalar-valued continuous homogeneous polynomial $P$ on $E$. A Banach space $E$ is {\it polynomially Schur} if every polynomially null $E$-valued sequence is norm null.

A subset $A$ of a topological vector space $E$ is {\it spaceable} (see \cite{aron1}) if there exists a closed infinite dimensional subspace of $E$ all of whose elements but the origin belong to $A$. The result due to Jim\'enez-Rodr\'iguez mentioned in the Introduction, which  is a spaceability result in $c_0^w(E)$, makes question (i) of the Introduction quite natural.
In Remark \ref{remfim} we shall explain why we cannot go to a space smaller than $c_0^w(E)$.

The {\it positive Schur property} in Banach lattices (positive -- or, equivalently,  disjoint -- weakly null sequences are norm null) was introduced by Wnuk \cite{wnuk1} and R\"abiger \cite{rabiger} and has been extensively studied, for some recent developments see \cite{arkadani, baklouti, botelho, chen, tradacete, wnuk4, zeekoei}. Question (ii) of the Introduction it nothing but the natural lattice counterpart of the problem solved by Jim\'enez-Rodr\'iguez in the Banach space setting.

A sequence $(x_j)_{j=1}^\infty$ in a Banach lattice $E$  is {\it regular-polynomially null} if $P(x_j) \longrightarrow 0$ for every scalar-valued regular homogeneous polynomial $P$ on $E$. The following class of Banach lattices was studied in \cite{joselucassegundo}: a Banach lattice $E$ is {\it positively polynomially Schur} if positive regular-polynomially null $E$-valued sequences are norm null. It is inevitable to wonder if the set of disjoint non-norm null regular-polynomially null  sequences in a non-positively polynomially Schur Banach lattice is (completely) latticeable in $\ell_\infty(E)$. This is the scope of question (iii).

In this section we show that the Jim\'enez-Rodr\'iguez technique can be adjusted to solve these three questions affirmatively. We will justify why we have to work with $\ell_\infty(E)$ in questions (ii) and (iii) (see Remark \ref{remfim}). We also establish the complete latticeability in $c_0(E)$ of the set of disjoint norm null sequences in arbitrary infinite dimensional Banach lattices.

Recall that $\ell_\infty(E)$ is a Banach lattice with the coordinatewise order whenever $E$ is a Banach lattice \cite[p.\,177]{aliprantis2} and that $c_0^w(E)$ is a closed subspace of $\ell_\infty(E)$ whenever $E$ is a Banach space \cite[p.\,33]{djt}.

\begin{theorem}\label{mainth} {\rm (a)} Let $E$ be a non-polynomially Schur Banach space. Then the set of $E$-valued non-norm null polynomially null sequences is spaceable in $c_0^w(E)$.

\medskip

\noindent {\rm (b)} Let $E$ be a Banach lattice failing the positive Schur property. Then the set of $E$-valued disjoint non-norm null weakly null sequences is completely latticeable in $\ell_\infty(E)$.

\medskip

\noindent{\rm (c)} Let $E$ be a non-positively polynomially Schur Banach lattice. Then the set of $E$-valued disjoint non-norm null regular-polynomially null sequences is completely latticeable in $\ell_\infty(E)$.
\end{theorem}

\begin{proof} We start with the construction due to Jim\'enez-Rodr\'iguez \cite{jimenez} which will be used in the three proofs. Let $E$ be a Banach space, $\varepsilon > 0$ and $(x_j)_{j=1}^\infty\subset B_E$ be a sequence such that $\|x_j\|\geq \varepsilon$ for every $j \in \mathbb{N}$. Consider the set of prime numbers ${\frak P} =\{p_k: k\in\mathbb{N}\}$ increasingly ordered, the surjective function
\begin{align*} F\colon \mathbb{N}\setminus\{1\}&\longrightarrow \mathbb{N}~,~f(m) = k,~{\rm where~} p_k=\min\{p \in {\frak P} : p\,|m\},
\end{align*}
and the map
\begin{align*}
    T\colon \ell_\infty\longrightarrow \ell_\infty(E)~,~T\left((a_n)_{n=1}^\infty\right)= \left(a_{F(j+1)}x_j\right)_{j=1}^\infty,
\end{align*}
that is,  $T\left((a_n)_{n=1}^\infty\right)_j=a_{F(j+1)}x_j$ for every $j\in\mathbb{N}$. An easy adaptation of the arguments of \cite[Theorem 2.1]{jimenez} yield that  $T$ is a well defined into isomorphism and that the nonzero elements of its range are non-norm null sequences. The range of $T$ will be the space/lattice we are looking for.

\medskip

\noindent (a) We can start with a non-norm null polynomially null sequence  $(x_j)_{j=1}^\infty$ in $E$. Passing to a subsequence and normalizing if necessary, we can suppose that $(x_j)_{j=1}^\infty\subset B_E$ and that there is $\varepsilon>0$ such that $\|x_j\|\geq \varepsilon$ for every $j\in\mathbb{N}$. All that is left to prove is that the elements of the range of $T$ are polynomially null $E$-valued sequences. This is true because, given $(a_{F(j+1)}x_j)_{j=1}^\infty\in T(\ell_\infty)$ with $(a_j)_{j=1}^\infty\in\ell_\infty$, $n\in\mathbb{N}$ and $P\in\mathcal{P}(^n E)$, since $P(x_j)\longrightarrow 0$ and $\{a_{F(j+1)} : j \in \mathbb{N}\}=\{a_n : n \in \mathbb{N}\}$ is a bounded set, we have  $$|P(a_{F(j+1)}x_j)|=|a_{F(j+1)}|^n\cdot|P(x_j)|\longrightarrow 0.$$
In particular, the range of $T$ lies in $c_0^w(E)$.

\medskip

\noindent(b) In this case we can start with a positive weakly null non-norm null sequence $(x_j)_{j=1}^\infty$ in $E$. By \cite[p.\,16]{wnuk} we can suppose that this sequence is disjoint and, as in the proof of (a), that  $(x_j)_{j=1}^\infty\subset B_E$ and that there is $\varepsilon>0$ such that $\|x_j\|\geq \varepsilon$ for every $j\in\mathbb{N}$. From the proof of \cite[Theorem 2.1]{jimenez} we know that the elements of the range of $T$ are weakly null sequences. As these elements are of the form $(a_{F(j+1)}x_j)_{j=1}^\infty$ for some $(a_n)_{n=1}^\infty\in\ell_\infty$ and the sequence $(x_j)_{j=1}^\infty$ is disjoint, from \cite[Lemma 1.9(1)]{aliprantis1} we conclude that the range of $T$ is formed by disjoint sequences. To finish the proof of this case, let us see that $T$ is Riesz homomorphism: given $(a_n)_{n=1}^\infty$ and $(b_n)_{n=1}^\infty$ in $\ell_\infty$, it holds $(a_n)_{n=1}^\infty\land(b_n)_{n=1}^\infty=(a_n\land b_n)_{n=1}^\infty$ and since $x_j \geq 0$ for every $j$, we have
$$(a_{F(j+1)}\land b_{F(j+1)})x_j=(a_{F(j+1)}x_j)\land (b_{F(j+1)}x_j)$$ for every $j\in\mathbb{N}$. It follows that $T((a_n)_{n=1}^\infty\land(b_n)_{n=1}^\infty)=T((a_n)_{n=1}^\infty)\land T((b_n)_{n=1}^\infty)$, proving that $T$ is a  Riesz homomorphism, hence its range is a closed sublattice of $\ell_\infty(E)$ lattice isomorphic to $\ell_\infty$.

\medskip

\noindent(c) According to \cite[Proposition 2.4]{joselucassegundo} we can start with a positive disjoint non-norm null regular-polynomially null sequence $(x_j)_{j=1}^\infty$ in $E$. Like we have done above, we can suppose that $(x_j)_{j=1}^\infty\subset B_E$ and that there is $\varepsilon>0$ such that $\|x_j\|\geq \varepsilon$ for every $j\in\mathbb{N}$. 
As we did in the proof of (b), the fact that the sequence $(x_j)_{j=1}^\infty$ is positive and disjoint guarantees that $T$ is a Riesz homomorphism, therefore its range is a closed sublattice of $\ell_\infty(E)$, and that the elements of the range are disjoint sequences. As we did in the proof of (a), the fact that the sequence $(x_j)_{j=1}^\infty$ is regular-polynomially null implies that the sequences in the range of $T$ are regular-polynomially null as well.
\end{proof}

Note that, although regular-polynomially null sequences are weakly null, Theorem \ref{mainth}(b) does not follow from (c) because there are positively polynomially Schur Banach lattices failing the positive Schur property, for instance, $\ell_p$ with $1< p < \infty$ (see \cite{joselucassegundo}).

It is natural to wonder if the closed sublattices of $\ell_\infty(E)$ obtained in the proofs of (b) and (c) above are ideals in $\ell_\infty(E)$. Next example shows that this is not the case in general, making clear that this is a direction that cannot be pursued using the Jim\'enez-Rodr\'iguez technique.

\begin{example}\rm The most favorable situation we can imagine for $T(\ell_\infty)$ to be an ideal of $\ell_\infty(E)$ occurs when the starting sequence $(x_j)_{j=1}^\infty$ is formed by atoms of $E$. In this example we show that, even in this case, $T(\ell_\infty)$ may fail to be an ideal of $\ell_\infty(E)$. We start with the sequence $(e_j)_{j=1}^\infty$ of canonical unit vectors in $c_0$, which is a positive disjoint non-norm null weakly null sequence formed by atoms. By \cite[Proposition 1.59]{dineen} this sequence is also regular-polynomially null. Consider the corresponding operator
\begin{align*}
    T\colon \ell_\infty\longrightarrow \ell_\infty(c_0)~,~T\left((a_n)_{n=1}^\infty\right)= \left(a_{F(j+1)}e_j\right)_{j=1}^\infty,
\end{align*}
the positive vector $(e_1,0,0,\ldots)\in\ell_\infty(c_0)$ and the sequence  $e_1\in\ell_\infty$. On the one hand,
$$0\leq(e_1,0,0,\ldots)\leq (e_1, 0, e_3, 0, e_5, 0,\ldots) = T(e_1)$$
  in $\ell_\infty(c_0)$.  On the other hand, there is no element  $(b_n)_{n=1}^\infty\in\ell_\infty$ such that $T((b_n)_{n=1}^\infty)=(e_1,0,0,\ldots)$. Indeed, supposing that such a sequence   $(b_n)_{n=1}^\infty$ exists, by the definition of $T$ we would have
$$(e_1, 0, 0, \ldots ) = T((b_n)_{n=1}^\infty) = (b_1e_1, b_2e_2, b_1e_3, \ldots), $$
which gives $1 = b_1 = 0$. This contradiction proves that
$T(\ell_\infty)$ is not an ideal in $\ell_\infty(E)$.
\end{example}

\begin{remark}\rm \label{remfim} 
(i) We cannot use $c_0^w(E)$ instead of $\ell_\infty(E)$ in Theorem \ref{mainth}(b) and (c) because $c_0^w(E)$ is not always a sublattice of $\ell_\infty(E)$. For instance, for $1 \leq p < \infty$, $c_0^w(L_p[0,1])$ is not a Riesz space due to the fact that the lattice operations in $L_p[0,1]$ are not weakly sequentially continuous \cite[Example, p.\,114]{meyer}.  But it is clear that the sublattices of $\ell_\infty(E)$ created in Theorem \ref{mainth}(b) and (c) are contained in $c_0^w(E)$. Sometimes $c_0^w(E)$ is a Banach lattice, for instante when $E$ is either an AM-space or an atomic Banach lattice with order continuous norm (see \cite[Theorem 12.30]{aliprantis2} and \cite[Proposition 2.5.23]{meyer}). In these cases, $\ell_\infty(E)$ can be replaced with $c_0^w(E)$ in Theorem \ref{mainth}(b) and (c).\\
(ii) Castillo, Garc\'ia and Gonzalo in \cite[Theorem 5.5]{castillo2} proved that the sum of two polynomially null sequences is not necessarily polynomially null. This is why we cannot pass to a space smaller than $c_0^w(E)$ in Theorem \ref{mainth}(a).\\
(iii) We have already explained why $c_0^w(E)$ cannot be used in general in Theorem \ref{mainth}(c). But one might wonder if we could have gone to a smaller space, formed by regular-polynomially null sequences. In order to see that we cannot, next we show that the counterxample given in \cite[Theorem 5.5]{castillo2} is good enough to show that the sum of two regular-polynomially null sequences may fail to be regular-polynomially null.
\end{remark}

\begin{theorem}\label{nespvet} The sum of two regular-polynomially null sequences in a Banach lattice is not necessarily regular-polynomially null.
\end{theorem}

\begin{proof}\rm Let $d(w;1)$ be the Lorentz space of \cite[Theorem 5.4]{castillo2} and denote by $d_*(w;1)$ its predual. The sequence of canonical unit vectors $(e_j)_{j=1}^\infty$ is a  1-unconditional basis for $d(w;1)$ (see \cite{albiac1}) and the sequence of coordinate functionals $(e^*_j)_{j=1}^\infty$ is an unconditional basis for  $d_*(w;1)$ (see \cite{galego}), hence it is a 1-unconditional basis (see \cite[I,\,p.\,19]{lindenstrauss}). We consider $d_*(w;1)$ as a Banach lattice with the order given by its 1-unconditional basis and $d(w;1)$ with its dual structure (which coincides, by the way, with the order given by the 1-unconditional basis $(e_j)_{j=1}^\infty$). Thus, $d_*(w;1)\times d(w;1)$ is a Banach lattice with the coordinatewise order, in which we can consider, without loss of generality, the norm $\|\cdot\|_1$. According to \cite[Theorem 5.5]{castillo2}, the sequences  $((e^*_j,0))_{j=1}^\infty$ and $((0,e_j))_{j=1}^\infty$ are polynomially null, hence regular-polynomially null, in $d_*(w;1)\times d(w;1)$. Let us see that their sum  $((e^*_j,e_j))_{j=1}^\infty$ is not regular-polynomially null. To do so, consider the symmetric bilinear form $A$ on $(d_*(w;1)\times d(w;1))\times (d_*(w;1)\times d(w;1))$ given by 
\begin{align*}
     A ((x^*,x),(y^*,y))= 1/2(x(y^*)+(y(x^*))
\end{align*}
(see \cite[Example 1.16]{dineen}). It is easy to check that $A$ is positive, from which it follows that  
its associated 2-homogeneous polynomial $\widehat{A}$ is positive, hence regular. Since $\widehat{A}((e^*_j,e_j))=1$ for every $j\in\mathbb{N}$, we conclude that $((e^*_j,e_j))_{j=1}^\infty$ is not regular-polynomially null.
\end{proof}
%
%
%
Now it is easy to see that, for every Banach space $E$, the set $PN$ of polynomially null $E$-valued sequences is spaceable in $c_0^w(E)$: if $E$ is not polynomially Schur, in Theorem \ref{mainth}(a) we proved that a set much smaller than $PN$ is spaceable; if $E$ is polynomially Schur, it is easy to check that $PN= c_0(E)$, the closed subspace of $c_0^w(E)$ formed by norm null sequences.

Theorem \ref{nespvet} rises the question of the complete latticeability of the set of disjoint regular-polynomially null sequences in a Banach lattice $E$. In Theorem \ref{mainth}(c) we proved that a set much smaller than this is completely latticeable whenever the Banach lattice $E$ is not positively polynomially Schur. We finish the section by giving a short proof that a set smaller than the set of disjoint regular-polynomially null sequences is completely latticeable in general. Recall that $c_0(E)$ is a closed sublattice of $\ell_\infty(E)$, hence a Banach lattice itself.

\begin{proposition} For every infinite dimensional Banach lattice $E$, the set of $E$-valued disjoint norm null sequences is completely latticeable in $c_0(E)$.
\end{proposition}

\begin{proof} By \cite[Lemma 2.6]{oikhberg} there is a positive disjoint sequence $(x_j)_{j=1}^\infty$ in $E$ such that $x_j \neq 0$ for every $j$. It is plain that
$$T \colon c_0 \longrightarrow c_0(E)~,~T\left((a_j)_{j=1}^\infty\right) = \left( \frac{a_j}{\|x_j\|}x_j\right)_{j=1}^\infty,$$
is a well defined linear isometric embedding. Since $x_j \geq 0$ for every $j$, the same reasoning of the proof of Theorem \ref{mainth}(b) shows that $T$ is a Riesz homomorphism, therefore $T(c_0)$ is an infinite dimensional closed sublattice of $c_0(E)$. The disjointness of the elements of the range of $T$ follows from the disjointness of the sequence $(x_j)_{j=1}^\infty$ combined with \cite[Lemma 1.9(1)]{aliprantis1}.
\end{proof}

\section{Complete latticeability of $\lambda_\pi(E) \setminus \lambda_s(E)$}\label{section3}

Given a Banach lattice $E$ and a space of scalar-valued sequences $\lambda$, the Banach lattice $\lambda_\pi(E)$ and its subspace $\lambda_s(E)$ were introduced in \cite{bu2} and have been studied in, e.g., \cite{botelho1,bu1,bu2,bu5,bu3}. The purpose of this section is to prove that $\lambda_\pi(E)\setminus\lambda_s(E)$ is either empty or completely latticeable in $\lambda_\pi(E)$ and to derive some consequences of this fact, especially when $\lambda$ is $\ell_p$ or an Orlicz sequence space. Let us recall the definitions of these spaces (see \cite{bu2}).

By a {\it sequence space} we mean a linear subspace $\lambda$ of $\mathbb{R}^\mathbb{N}$. The {\it K\"othe dual} of $\lambda$, which is also a sequence space, is defined by
$$
\lambda'=\left\{(b_j)_{j=1}^\infty\in\mathbb{R}^\mathbb{N}: \sum_{j=1}^\infty|a_jb_j|<\infty {\rm ~for~every~}  (a_j)_{j=1}^\infty\in\lambda\right\}.
$$


Suppose that $\lambda$ is an order continuous Banach lattice. Then $\lambda'=\lambda^*$ (\cite[p.\,339]{bu2} or \cite[II,\,p.\,29]{lindenstrauss})  and $\lambda'$ is a Banach lattice with the norm
$$
\|(b_j)_{j=1}^\infty\|_{\lambda'}=\|(b_j)_{j=1}^\infty\|_{\lambda^*}=\sup_{(a_j)_{j=1}^\infty\in B_\lambda}\left|\sum_{j=1}^\infty a_jb_j\right|.
$$

We are interested in sequence spaces enjoying some special properties, which, as we shall see, include the most important examples.

\begin{definition}\label{deffed}\rm (a) For a sequence $x=(x_j)_{j=1}^\infty$ in a linear space, we define $x^0:=(x_j^0)_{j=1}^\infty$, where $x^0_j$ is the $j$-th nonzero coordinate of $x$ if such coordinate exists and zero otherwise. \\
(b) A sequence space $\lambda$ is an {\it invariant sequence lattice} if it is a  KB-space with $\|e_j\|_{\lambda}=1$ for every $j\in\mathbb{N}$ such that $\lambda'$ enjoys the following conditions:

(i) $x\in\lambda'\Leftrightarrow x^0\in\lambda'$ and, in this case, $\|x\|_{\lambda'}=\|x^0\|_{\lambda'}$;

(ii) if $x\in\lambda'$ then every subsequence $z$ of $x$ belongs to $\lambda'$ and $\|z\|_{\lambda'}\leq\|x\|_{\lambda'}$.
\end{definition}

Note that the definiton of $x^0$ is slightly different from the one given in \cite{geraldo} and coincides with the vector $x'$ ({\it closing up $x$}) given in  \cite{carando}.

It is not difficult to check that if $\lambda$ is a Banach lattice such that $\|e_j\|_\lambda=1$ for every $j\in\mathbb{N}$, then $\|e_j\|_{\lambda'}=1$ for every $j\in\mathbb{N}$.

\begin{examples}\label{exampless}\rm  (a) If  $\lambda$ is such that $\lambda'$ is a symmetric sequence space in the sense of \cite{carando}, then $\lambda'$ enjoys the conditions (i) and (ii) above \cite[Proposition 2.2]{carando}.\\
(b) For $1\leq p<\infty$, $\ell_p$ is  a KB-space. It is clear that $\|e_j\|_p = 1$ and that $\ell'_p=\ell_{p'}$ (where $1/p+1/p'=1$) is symmetric, hence $\ell_p$ is an invariant sequence lattice for $1\leq p<\infty$.\\
(c) Let $\varphi \colon [0,\infty) \longrightarrow [0, \infty)$ be an Orlicz function satisfying the $\Delta_2$-condition of \cite[I, Definition 4.a.3]{lindenstrauss} and admitting a complementary function $\varphi^*$ according to \cite[I, p.\,147]{lindenstrauss},  and let $\ell_\varphi$ be the corresponding Orlicz sequence space, which happens to be a Banach lattice. The K\"othe dual of $\ell_\varphi$ is $\ell_{\varphi^*}$ \cite[Corollary 8.28]{gupta} and it is plain that $\ell_{\varphi^*}$ is symmetric. Since $\ell_\varphi$ is a KB-space (see the proof of \cite[I, Theorem 4.a.9]{lindenstrauss}) and $\|e_j\|_{\ell_\varphi} = 1$, then $\ell_\varphi$ is an invariant sequence lattice.
\end{examples}

Now we present, according to \cite{bu2}, the definitions of the vector-valued sequence spaces we shall work with.

\begin{definition}\rm For a Banach lattice $E$ and a KB-space $\lambda$ such that  $\|e_j\|_\lambda=1$ for every  $j\in\mathbb{N}$,
$$
\lambda'_w(E^*):=\left\{(x^*_j)_{j=1}^\infty\in (E^*)^\mathbb{N}:  (x^*_j(x))_{j=1}^\infty\in\lambda' {\rm ~for~every~}x\in E\right\}
$$
and
$$
\lambda_s(E):=\left\{(x_j)_{j=1}^\infty\in E^\mathbb{N}:  \sum_{j=1}^\infty|x^*_j(x_j)|<\infty  {\rm ~for~every~} (x^*_j)_{j=1}^\infty\in \lambda'_w(E^*) \right\}$$
are Banach spaces (not necessarily Banach lattices) with the norms
$$
\|(x^*_j)_{j=1}^\infty\|_w:=\sup_{x\in B_E}\|(x^*_j(x))_{j=1}^\infty\|_{\lambda'}~,~
\|(x_j)_{j=1}^\infty\|_s:=\sup_{(x^*_j)_{j=1}^\infty\in B_{\lambda'_w(E^*)}}\left|\sum_{j=1}^\infty x^*_j(x_j)\right|,
$$
respectively \cite[p.\,339]{bu2}. Also,
$$
\lambda'_\varepsilon(E^*):=\left\{(x^*_j)_{j=1}^\infty\in (E^*)^\mathbb{N}: (|x^*_j|(x))_{j=1}^\infty\in\lambda' {\rm ~for~every~} x\in E^+\right\}
$$
and
$$
\lambda_\pi(E):=\left\{(x_j)_{j=1}^\infty\in E^\mathbb{N} : \sum_{j=1}^\infty x^*_j(|x_j|)<\infty {\rm ~for~every~} (x^*_j)_{j=1}^\infty\in \lambda'_\varepsilon(E^*)^+ \right\}
$$
are Banach lattices with the norms
$$
\|(x^*_j)_{j=1}^\infty\|_\varepsilon:=\sup_{x\in B_{E^+}}\|(|x^*_j|(x))_{j=1}^\infty\|_{\lambda'}~,~
\|(x_j)_{j=1}^\infty\|_\pi:=\sup_{(x^*_j)_{j=1}^\infty\in B_{\lambda'_\varepsilon(E^*)^+}}\sum_{j=1}^\infty x^*_j(|x_j|),
$$
respectively \cite[p.\,344]{bu2}.
\end{definition}

In the examples of the theory developed so far about these spaces (see \cite{botelho1,bu1,bu2,bu5,bu3}, the underlying sequence space $\lambda$ is usually one of the invariant sequence lattices listed in Examples \ref{exampless}.

For $\lambda = \ell_p$, $1 < p < \infty$, $\lambda_s(E)$ coincides isometrically with the space $\ell_p\langle E\rangle$ of Cohen strongly $p$-summable sequences (see \cite[p.\,520]{bu4} and \cite[p.\,339]{bu2}), which is isometrically isomorphic to the completed projective tensor product $\ell_p\widehat{\otimes}_{\pi}E$ \cite[p.\,525]{bu4}. And, in this case, $\lambda_\pi(E)$ is lattice isometric to the positive projective tensor product $\ell_p\widehat{\otimes}_{|\pi|}E$ \cite[Theorem 15]{bu1}. For $\lambda = \ell_p$, $1 < p < \infty$, we shall henceforth write $\ell_p\langle E\rangle$ instead of $\lambda_s(E)$ and $\ell_p^\pi(E)$ instead of $\lambda_\pi(E)$.

According to \cite[Proposition 5.2]{bu2}, for every KB-space $\lambda$ such that $\|e_j\|_\lambda=1$ and every Banach lattice $E$, $\lambda_s(E)$ is a linear subspace of $\lambda_\pi(E)$ and the inclusion is a norm one operator. This rises the question of whether or not $\lambda_\pi(E)\setminus\lambda_s(E)$ is (completely) latticeable. The purpose of this section is to solve this question affirmatively.

Before solving the problem, let us justify why the (few) known general criteria for complete latticeability, due to Oikhberg, do not apply to our situation, even in the canonical case $\lambda = \ell_p$, $1 < p < \infty$. We need the following lemma.

\begin{lemma}\label{lemmaidfech} For $1< p,q < \infty$, $\ell_p\langle \ell_q\rangle$ is not a closed ideal of $\ell^\pi_p(\ell_q)$.
\end{lemma}

\begin{proof} Suppose that $\ell_p\langle \ell_q\rangle$ is a closed ideal of $\ell^\pi_p(\ell_q)$. In this case $(\ell_p\langle \ell_q\rangle,\|\cdot\|_\pi)$ is a Banach lattice and, since $\|\cdot\|_\pi\leq\|\cdot\|_s$,  the identity operator $(\ell_p\langle \ell_q\rangle,\|\cdot\|_s)\longrightarrow (\ell_p\langle \ell_q\rangle,\|\cdot\|_\pi)$ is a continuous bijection. By the Open Mapping Theorem this identity is an isomorphism. As we have already mentioned, $(\ell_p\langle \ell_q\rangle,\|\cdot\|_s)$ is isomorphic to $\ell_p\widehat{\otimes}_\pi \ell_q$, hence the latter space is also isomorphic to the Banach lattice $(\ell_p\langle \ell_q\rangle,\|\cdot\|_\pi)$. Therefore, $(\ell_p\widehat{\otimes}_\pi \ell_q)^{**}$  is isomorphic to the Banach lattice $(\ell_p\langle \ell_q\rangle,\|\cdot\|_\pi)^{**}$. From \cite[p.\,59]{johnson} (or \cite[Theorem 17.5]{djt}) it follows that $\ell_p\widehat{\otimes}_\pi \ell_q$ has Gordon-Lewis local unconditional structure. But  $\ell_p\widehat{\otimes}_\pi \ell_q$ does not have Gordon-Lewis local unconditional structure by \cite[Corollary 3.6]{gordon}.
\end{proof}

Now we return to the general criteria for complete latticeability.

\medskip

\noindent $\bullet$ \cite[Corollary 1.4]{oikhberg2} does not apply to $\lambda_\pi(E)\setminus\lambda_s(E)$ because $\lambda_s(E)$ is not the range of a compact operator on $\lambda_\pi(E)$. Otherwise  $\lambda_s(E)$ would be separable \cite[Proposition 3.4.7]{megginson}, and this is not the case for a nonseparable $E$ because  $\lambda_s(E)$ contains an isometric copy of $E$.\\ 
$\bullet$ \cite[Proposition 2.9]{oikhberg} does not fit because the order in $\lambda_\pi(E)$ is not given by an 1-unconditional basis in general. Actually, $\lambda_\pi(E)$ may be nonseparable since it contains a copy of $E$.\\
$\bullet$ \cite[Theorem 1.1]{oikhberg2} cannot be applied because $\lambda_\pi(E)$ is not always order continuous. In fact, taking $E = \ell_\infty$, which is  $\sigma$-Dedekind complete, and $\lambda = \ell_p$, $1 < p < \infty$, $\lambda_\pi(E)$ is $\sigma$-Dedekind complete as well \cite[Theorem 5.5]{bu2} and, as $\lambda_\pi(E)$ contains a copy of $E$, it follows from \cite[Theorem 14.9]{aliprantis2}  that $\lambda_\pi(E)$ fails to be order continuous.\\
$\bullet$  Finally, \cite[Proposition 2.4]{oikhberg} does not apply either because, due to Lemma \ref{lemmaidfech}, $\lambda_s(E)$ is not a closed ideal of $\lambda_\pi(E)$ in general.

Once the general criteria do not apply, an {\it ad hoc} argument is needed to prove the complete latticeability of $\lambda_\pi(E)\setminus\lambda_s(E)$. Our proof is based on the mother vector technique developed in \cite{geraldo}.

The notation $x=\sum\limits_{j=1}^\infty x_j\cdot e_j$ shall henceforth mean that $x$ is a sequence whose $j$-th coordinate is $x_j$, $j \in \mathbb{N}$.
\begin{lemma}\label{lemal} Let $\lambda$ be a KB-space such that $\|e_j\|_\lambda=1$ and let $E$ be a Banach lattice. Then $\lambda_\pi(E) \subseteq \ell_\infty(E)$ and $\|\cdot\|_\infty \leq \|\cdot\|_\pi$.
\end{lemma}

\begin{proof} Given $(y_j)_{j=1}^\infty\in\lambda_\pi(E)$, let $i\in\mathbb{N}$ be such that $y_i\neq 0$. By the Hahn-Banach Theorem there is $y^*\in E^*$ such that $\|y^*\|_{E^*}=1$ and $y^*(|y_i|)=\||y_i|\|_E=\|y_i\|_E$. The sequence $(\psi_j)_{j=1}^\infty:=|y^*|\cdot e_i$ belongs to $B_{\lambda'_\varepsilon(E^*)^+}$ because
\begin{align*}
\|(\psi_j)_{j=1}^\infty\|_\varepsilon&=\sup_{x\in B_{E^+}}\|(|\psi_j|(x))_{j=1}^\infty\|_{\lambda'}=\sup_{x\in B_{E^+}}\|(|y^*|(x))e_i\|_{\lambda'}\\&=\sup_{x\in B_{E^+}}(|y^*|(x))\|e_i\|_{\lambda'}= \||y^*|\|_{E^*}=\|y^*\|_{E^*}=1.
\end{align*}
Therefore, for every $i \in \mathbb{N}$,
\begin{align*}
\|y_i\|_E&=y^*(|y_i|)=|y^*(|y_i|)|{\leq} |y^*|(|y_i|)=\sum_{j=1}^\infty\psi_j(|y_j|)\\& \leq \sup_{(y^*_j)_{j=1}^\infty\in B_{\lambda'_\varepsilon(E^*)^+}}\sum_{j=1}^\infty y^*_j(|y_j|)=\|(y_j)_{j=1}^\infty\|_\pi,
\end{align*}
from which the result follows.
\end{proof}

\begin{lemma}\label{llema} Let $\lambda$ be an invariant sequence lattice and let $E$ be a Banach lattice. Then:\\
{\rm (a)} $\lambda'_\varepsilon(E^*)$ and $\lambda'_w(E^*)$ satisfy conditions {\rm (i)} and {\rm (ii)} of Definition \ref{deffed}(b).\\
{\rm (b)} $\lambda_\pi(E)$ and $\lambda_s(E)$ satisfy condition {\rm (i)} of Definition \ref{deffed}(b).
\end{lemma}

\begin{proof} (a) Let $x^* = (x_j^*)_{j=1}^\infty$. On the one hand, from the definitions and the properties of $\lambda$,
 $$x^* \in \lambda_\varepsilon'(E^*) \Leftrightarrow \left(|x_j^*|(x) \right)_{j=1}^\infty \in \lambda', \forall x \in E^+ \Leftrightarrow \left(\left(|x_j^*|(x) \right)_{j=1}^\infty\right)^0 \in \lambda', \forall x \in E^+ {\rm ~and~}
$$
$$\|x^*\|_{\varepsilon}  = \sup_{ x \in B_{E^+}}\left\|\left(|x_j^*|(x) \right)_{j=1}^\infty\right\|_{\lambda'}  = \sup_{x \in B_{E^+}}\left\|\left(\left(|x_j^*|(x) \right)_{j=1}^\infty\right)^0\right\|_{\lambda'}.$$
On the other hand, for the same reasons,
$$
(x^*)^0 \in \lambda_\varepsilon'(E^*) \Leftrightarrow \left(|(x_j^*)^0|(x) \right)_{j=1}^\infty \in \lambda', \forall x \in E^+ \Leftrightarrow \left(\left(|(x_j^*)^0|(x)\right)_{j=1}^\infty\right)^0 \in \lambda', \forall x \in E^+ {~\rm and}
$$
\begin{align*}
\|(x^*)^0\|_{\varepsilon}  = \sup_{x \in B_{E^+}}\left\|\left(|(x_j^*)^0|(x) \right)_{j=1}^\infty\right\|_{\lambda'}
 = \sup_{x \in B_{E^+}}\left\|\left(\left(|(x_j^*)^0|(x) \right)_{j=1}^\infty\right)^0\right\|_{\lambda'}.
\end{align*}
A moment's reflection yields that
\begin{equation*}\label{eq.zerofree}
\left(\left(|x_j^*|(x) \right)_{j=1}^\infty\right)^0=\left(\left(|(x_j^*)^0|(x) \right)_{j=1}^\infty\right)^0.
\end{equation*}
This shows that $x^*\in\lambda'_\varepsilon(E^*)\Leftrightarrow (x^*)^0\in\lambda'_\varepsilon(E^*)$ and  $\|x^*\|_\varepsilon=\|(x^*)^0\|_\varepsilon$, proving that $\lambda'_\varepsilon(E^*)$ satisfies condition {\rm (i)} of Definition \ref{deffed}(b). The case of $\lambda'_w(E^*)$ is similar.

Given $x^*\in\lambda'_\varepsilon(E^*)$, the fact that $\lambda$ is an invariant sequence lattice implies immediately that every subsequence $z^*$ of $x^*$ belongs to  $\lambda'_\varepsilon(E^*)$. Moreover, $\|(|z^*_j|(x))_{j=1}^\infty\|_{\lambda'}\leq \|(|x^*_j|(x))_{j=1}^\infty\|_{\lambda'}$ for $x\in B_{E^+}$, so the definition of  $\|\cdot\|_\varepsilon$ gives $\|z^*\|_\varepsilon\leq\|x^*\|_\varepsilon$, showing that $\lambda'_\varepsilon(E^*)$ satisfies condition {\rm (ii)} of Definition \ref{deffed}(b). Again, the case of $\lambda'_w(E^*)$ is similar.

\medskip

\noindent (b) Let $x=(x_j)_{j=1}^\infty\in\lambda_\pi(E)$ be given and let $x^0=(x_{n_j})_{j=1}^\infty$, that is, $x_{n_j}$ is the $j$-th nonzero coordinate of  $x$ if such coordinate exists and zero otherwise.  Given  $x^*=(x^*_j)_{j=1}^\infty\in\lambda'_\varepsilon(E^*)^+$, considere the sequence  
$y^*=\sum\limits_{j=1}^\infty x^*_j\cdot e_{n_j}$. Since $(y^*)^0=(x^*)^0$ and $x^*\in\lambda'_\varepsilon(E^*)^+$, by (a) we have  $y^*\in\lambda'_\varepsilon(E^*)^+$. Thus, $\sum\limits_{j=1}^\infty x^*_j(|x_{n_j}|)=\sum\limits_{j=1}^\infty y^*_j(|x_j|)<\infty$, showing that $x^0\in\lambda_\pi(E)$. Given $x^*\in B_{\lambda'_\varepsilon(E^*)^+}$, since
 $$\|y^*\|_\varepsilon=\|(y^*)^0\|_\varepsilon=\|(x^*)^0\|_\varepsilon=\|x^*\|_\varepsilon,$$
 we have $y^*\in B_{\lambda'_\varepsilon(E^*)^+}$. Therefore,
\begin{equation*}\label{ineq.norm1}
\|x^0\|_\pi=\sup_{x^*\in B_{\lambda_\varepsilon(E)^+}}\sum_{j=1}^\infty x^*_j(|x_{n_j}|)\leq \sup_{y^*\in B_{\lambda_\varepsilon(E)^+}}\sum_{j=1}^\infty y^*_j(|x_j|)=\|x\|_\pi.
\end{equation*}
Conversely, let $x=(x_j)_{j=1}^\infty\in E^\mathbb{N}$ be such that $x^0=(x_{n_j})_{j=1}^\infty\in\lambda_\pi(E)$. 
Given $x^*=(x^*_j)_{j=1}^\infty\in\lambda'_\varepsilon(E^*)^+$, consider the subsequence 
$z^*=\sum\limits_{j=1}^\infty x^*_{n_j}\cdot e_j$ of $x^*$. By (a) we know that $z^*\in\lambda'_\varepsilon(E^*)^+$, then $\sum\limits_{j=1}^\infty x^*_j(|x_j|)=\sum\limits_{j=1}^\infty z^*_j(|x_{n_j}|)<\infty$, proving that $x\in\lambda_\pi(E)$. If  $x^*\in B_{\lambda_\varepsilon(E^*)^+}$, again by (a) we have that $z^*\in B_{\lambda_\varepsilon(E^*)^+}$, so
\begin{equation*}\label{ineq.norm2}
\|x\|_\pi=\sup_{x^*\in B_{\lambda'_\varepsilon(E^*)^+}}\sum_{j=1}^\infty x^*_j(|x_j|)
\leq\sup_{z^*\in B_{\lambda'_\varepsilon(E^*)^+}}\sum_{j=1}^\infty z^*_j(|x_{n_j}|)=\|x^0\|_\pi,
\end{equation*}
which proves that $\lambda_\pi(E)$ satisfies condition {\rm (i)} of Definition \ref{deffed}(b). The case of $\lambda_s(E)$ is similar.
\end{proof}


\begin{theorem}\label{mth} Let $E$ be a Banach lattice and let $\lambda$ be an invariant sequence lattice. Then $\lambda_\pi(E)\setminus \lambda_s(E)$ is either empty or completely latticeable.
\end{theorem}

\begin{proof} Suppose that $\lambda_\pi(E)\setminus\lambda_s(E)\neq\emptyset$ and pick $x\in\lambda_\pi(E)\setminus \lambda_s(E)$. Writing $x=x^+-x^-$, as $\lambda_s(E)$ is a linear subspace, $x^+$ or $x^-$ does not belong to $\lambda_s(E)$, meaning that we can suppose $x$ to be positive, say $x=(x_j)_{j=1}^\infty\in\lambda_\pi(E)^+\setminus \lambda_s(E)$. Since  $x\notin\lambda_s(E)$, the set $\{j\in\mathbb{N};~x_j\neq 0\}$ is infinite. By Lemma \ref{llema}, $x\in\lambda_\pi(E)\Leftrightarrow x^0\in\lambda_\pi(E)$ and $x\in\lambda_s(E)\Leftrightarrow x^0\in\lambda_s(E)$, so we can assume, without loss of generality, that $x_j\neq 0$ for every $j\in\mathbb{N}$.

Split $\mathbb{N}=\textstyle\bigcup\limits_{i=1}^\infty\mathbb{N}_i$ into infinitely many pairwise disjoint subsets $\mathbb{N}_i=\{i_1<i_2<\cdots\}$, $i \in \mathbb{N}$. Define $y_i=\sum\limits_{j=1}^\infty x_j\cdot e_{i_j}$ for every $i\in\mathbb{N}$ and consider the operator
$$T\colon\ell_1\longrightarrow \lambda_\pi(E)~,~ T\left(
    (a_i)_{i=1}^\infty\right) = \sum_{i=1}^\infty a_iy_i.$$
Each $y_i^0=x$, so $y_i\in\lambda_\pi(E)^+\setminus \lambda_s(E)$ for every $i\in\mathbb{N}$ by Lemma \ref{llema}. For $(a_i)_{i=1}^\infty\in\ell_1$,
$$
\sum_{i=1}^\infty\|a_iy_i\|_\pi=\sum_{i=1}^\infty|a_i|\cdot \|y_i^0\|_\pi= \|x\|_\pi \cdot \sum_{i=1}^\infty|a_i|<\infty,
$$
showing that $T$ is well defined because $\lambda_\pi(E)$ is a Banach space. It is obvious that $T$ is linear, let us show its injectivity: for any $(a_i)_{i=1}^\infty\in\ell_1$,
$$
T\left((a_i)_{i=1}^\infty\right)= \sum_{i=1}^\infty a_iy_i=\sum_{i=1}^\infty a_i\left(\sum_{j=1}^\infty x_j\cdot e_{i_j}\right)=  \sum_{i=1}^\infty\left(\sum_{j=1}^\infty a_ix_j\cdot e_{i_j}\right).
$$
Since the sets $\mathbb{N}_i$ are pairwise disjoint, each coordinate of $T\left((a_i)_{i=1}^\infty\right)$ is of the form $a_ix_j$ for some $i,j\in\mathbb{N}$. Since $x_j\neq 0$ for every $j\in\mathbb{N}$, we have
$T\left((a_i)_{i=1}^\infty\right)=0$ if and only if $(a_i)_{i=1}^\infty=0$.

Let us check now that $T$ is a Riesz homomorphism: it is clear that $y_i\perp y_k$ para all $i\neq k$, so, for any $(a_i)_{i=1}^\infty\in\ell_1$,
\begin{align*}
T\left(\left|(a_i)_{i=1}^\infty\right|\right)&=\sum_{i=1}^\infty |a_i|y_i \stackrel{(\star)}{=}\lim_{n\rightarrow\infty}
\sum_{i=1}^n|a_iy_i|\stackrel{(\star\star)}{=}
\lim_{n\rightarrow\infty}\left|\sum_{i=1}^na_iy_i\right|=\left|\sum_{i=1}^\infty a_iy_i\right|=\left|T\left((a_i)_{i=1}^\infty\right)\right|,
\end{align*}
where ($\star$) follows from the fact that each $y_i\geq 0$ and ($\star\star$) from \cite[Lemma 1.9]{aliprantis1}. 
It follows that $\overline{T(\ell_1)}$ is an infinite dimensional closed sublattice of $\lambda_\pi(E)$.

All that is left to prove is that $\overline{T(\ell_1)}\cap \lambda_s( E)=\{0\}.$  Given $0\neq z\in\overline{T(\ell_1)}$, take $(a_i^{(k)})_{i=1}^\infty\in\ell_1$, $k\in\mathbb{N}$, such that $\lim\limits_{k\rightarrow \infty} T\left((a_i^{(k)})_{i=1}^\infty\right)=z$ in $\lambda_\pi(E)$. For each $k\in\mathbb{N}$,
\begin{equation}\label{eq.imT}
T\left((a_i^{(k)})_{i=1}^\infty\right)= \sum_{i=1}^\infty a_i^{(k)}y_i=\sum_{i=1}^\infty \left(a_i^{(k)}\sum_{j=1}^\infty x_j\cdot e_{i_j}\right)=  \sum_{i=1}^\infty\left(\sum_{j=1}^\infty a^{(k)}_ix_j\cdot e_{i_j}\right).
\end{equation}
Taking $r\in\mathbb{N}$ such that $z_r\neq 0$, there are unique $m,t\in\mathbb{N}$ so that $e_{m_t}=e_r$. Consider the set $\mathbb{N}_m=\{m_1<m_2<\cdots\}$ and note that from (\ref{eq.imT}) it follows that, for all $j,k\in\mathbb{N}$ the $m_j$-th coordinate of $T\left((a^{(k)}_i)_{i=1}^\infty\right)$ is equal to $a_m^{(k)}x_j$. From Lemma \ref{lemal} we know that convergence in $\lambda_\pi(E)$ implies coordinatewise convergente, so $z_{m_j}=\lim\limits_{k\rightarrow\infty}a_m^{(k)}x_j=\left(\lim\limits_{k\rightarrow\infty}a_m^{(k)}\right)x_j$, for every $j\in\mathbb{N}$. Writing $a_m=\lim\limits_{k\rightarrow\infty}a_m^{(k)}$ we get $z_{m_j}=a_mx_j$ for every $j\in\mathbb{N}$ and $a_m\neq 0$ since $a_mx_t=z_{m_t}=z_r\neq 0$.

Once $x\not\in\lambda_s(E)$, there exists  $(\varphi_j)_{j=1}^\infty\in\lambda'_w(E^*)$ such that $\sum\limits_{j=1}^\infty|\varphi_j(x_j)|=\infty$. Defining $\psi=\sum_{j=1}^\infty\varphi_j\cdot e_{m_j} \in \lambda'_w(E^*)$, we have
$$
\sum_{j=1}^\infty|\psi_j(z_j)|= \sum_{j=1}^\infty|\psi_{m_j}(z_{m_j})|=\sum_{j=1}^\infty|\varphi_j(a_m x_j)|=|a_m|\cdot\sum_{j=1}^\infty|\varphi_j(x_j)|=\infty,
$$
proving that $z\not\in\lambda_s(E)$.
\end{proof}


When $\lambda = \ell_\varphi$, where $\varphi$ is an Orlicz function, we write $\ell^\pi_\varphi(E):= \lambda_\pi(E)$ and $\ell^s_\varphi(E):= \lambda_s(E)$.   Combining Example \ref{exampless}(c) with Theorem \ref{mth} we get the following.

\begin{corollary}\label{ret.orlicz} Let $E$ be a Banach lattice and $\varphi$ be an Orlicz function satisfying the $\Delta_2$-condition and admitting a  complementary function. Then $\ell^\pi_\varphi(E)\setminus\ell^s_\varphi(E)$ is either empty or completely latticeable.
\end{corollary}

For $\lambda = \ell_p$, $1 < p < \infty$, we can go a bit further.

\begin{corollary}\label{fcorf} Let $E$  be a Banach lattice not isomorphic to an AL-space and  $1< p<\infty$. Then $\ell^\pi_p(E)\setminus\ell_p\langle E\rangle$ is completely latticeable.
\end{corollary}

\begin{proof} Bearing Example \ref{exampless}(c) and Theorem \ref{mth} in mind, it is enough to check that $\ell_p\langle E \rangle \neq \ell_p^\pi(E)$. Since the inclusion operator $i \colon \ell_p\langle E \rangle \hookrightarrow \ell_p^\pi(E)$ is continuous, supposing $\ell_p\langle E \rangle =  \ell_p^\pi(E)$ we have by the Open Mapping Theorem that the two spaces are isomorphic. As mentioned earlier, $\ell_p\langle E \rangle$ is isomorphic to $\ell_p\widehat{\otimes}_{\pi}E$ and $\ell_p^\pi(E)$ is (lattice) isomorphic to $\ell_p\widehat{\otimes}_{|\pi|}E$, so the Banach spaces $\ell_p\widehat{\otimes}_{\pi}E$ and $\ell_p\widehat{\otimes}_{|\pi|}E$ are isomorphic. From \cite[Proposition 42]{bu5} it follows that $\ell_p$ or $E$ is isomorphic to an AL-space. This contradiction completes the proof.
\end{proof}

Let us stress how we are going to regard $\ell_p\widehat{\otimes}_{\pi}E$ as a linear subspace of $\ell_p\widehat{\otimes}_{|\pi|}E$. It is clear that
$$A \colon \ell_p \times E \longrightarrow \ell_p\widehat{\otimes}_{|\pi|}E~,~A(t,y) = t\otimes y, $$
is a continuous bilinear operator. So, its linearization
$$A_L \colon \ell_p\widehat{\otimes}_{\pi}E \longrightarrow \ell_p\widehat{\otimes}_{|\pi|}E$$
is a bounded linear operator such that $A_L(t \otimes x) = t \otimes x$ for all $t \in \ell_p$ and $x \in E$. Both the isometric isomorphism $ S \colon \ell_p\widehat{\otimes}_{\pi}E {\longrightarrow} \ell_p\langle E\rangle$ from \cite{bu4} and the lattice isometric isomorphism $ T \colon \ell_p\widehat{\otimes}_{|\pi|}E {\longrightarrow} \ell_p^\pi( E)$ from \cite{bu1} send $(t_j)_{j=1}^\infty \otimes x$ to $(t_jx)_{j=1}^\infty$. So, $A_L$ and $T^{-1} \circ i \circ S$ coincide on the elementary tensors $t \otimes x$. The uniqueness of the linearization of a continuous bilinear operator gives that $A_L =T^{-1} \circ i \circ S$, from which it follows that $A_L$ is injective. Therefore we can identify, as vector spaces, $\ell_p\widehat{\otimes}_{\pi}E$ with the linear subspace $A_L(\ell_p\widehat{\otimes}_{\pi}E)$ of $\ell_p\widehat{\otimes}_{|\pi|}E$ in such a way that $A_L$ restricted to $\ell_p \otimes E$ is the identity operator. 

The following lemma is straightforward.

\begin{lemma}\label{ammel} Let $u \colon E \longrightarrow F$ be a lattice isomorphism (Riesz homomorphism $+$ Banach space isomorphism) between Banach lattices, let $X$ be a subset of $E$ and $Y$ be a subset of $F$ so that $Y\subset u(X)$. If $E\setminus X$ is completely latticeable then $F\setminus Y$ is completely latticeable too.
\end{lemma}


Regarding $\ell_p\widehat{\otimes}_{\pi}E$ as a linear subspace of $\ell_p\widehat{\otimes}_{|\pi|}E$ as above, we have:

\begin{corollary} Let $E$  be a Banach lattice not isomorphic to an AL-space and  $1< p<\infty$. Then $(\ell_p\widehat{\otimes}_{|\pi|}E)\setminus(\ell_p\widehat{\otimes}_{\pi}E)$ is completely latticeable.
\end{corollary}

\begin{proof} Let $i$, $A_L$, $S$ and $T$ be the operators of the paragraph before the lemma. Then $T^{-1} \colon \ell_p^\pi(E) \longrightarrow \ell_p\widehat{\otimes}_{|\pi|}E$ is a lattice isometric isomorphism, $\ell_p\langle E\rangle$ is a subspace of $\ell_p^\pi(E)$, $\ell^\pi_p(E)\setminus\ell_p\langle E\rangle$ is completely latticeable by Theorem \ref{fcorf} and $A_L(\ell_p\widehat{\otimes}_{\pi}E)$ is the subspace of $\ell_p\widehat{\otimes}_{|\pi|}E$ we are identifying with $\ell_p\widehat{\otimes}_{\pi}E$. Since
$$A_L(\ell_p\widehat{\otimes}_{\pi}E) = (T^{-1} \circ i \circ S)(\ell_p\widehat{\otimes}_{\pi}E) = T^{-1}(S(\ell_p\widehat{\otimes}_{\pi}E))= T^{-1}(\ell_p\langle E \rangle), $$
Lemma \ref{ammel} gives the complete latticeability of $(\ell_p\widehat{\otimes}_{|\pi|}E)\setminus A_L(\ell_p\widehat{\otimes}_{\pi}E)$.
\end{proof}

\bigskip

\noindent Faculdade de Matem\'atica~~~~~~~~~~~~~~~~~~~~~~Departamento de Matem\'atica\\
Universidade Federal de Uberl\^andia~~~~~~~~ IMECC-UNICAMP\\
38.400-902 -- Uberl\^andia -- Brazil~~~~~~~~~~~~ 13.083-859 -- Campinas -- Brazil\\
e-mail: botelho@ufu.br ~~~~~~~~~~~~~~~~~~~~~~~~~e-mail: lucasvt09@hotmail.com

\end{document}